\documentclass[10pt]{article}
\usepackage[english]{babel}
\usepackage[margin=1.5in]{geometry}
\usepackage{amsmath, amsfonts, amssymb, amsthm}
\usepackage[hidelinks]{hyperref}
\usepackage[scr=boondoxo,cal=euler]{mathalfa}
\usepackage{tikz-cd}
\usepackage{pxfonts}
\usepackage{indentfirst}
\usepackage{fancyhdr}
\usepackage{graphicx}
\usepackage{enumitem}
\usepackage{microtype}
\usepackage{pdfpages}
\usepackage{centernot}
\usepackage{ragged2e}
\usepackage{mathabx}
\allowdisplaybreaks

\usetikzlibrary{decorations.pathmorphing}

\theoremstyle{plain}
\newtheorem{thm}{Theorem}[section]
\newtheorem*{thm*}{Theorem}
\newtheorem{lem}[thm]{Lemma}

\newtheorem*{prop*}{Proposition}

\theoremstyle{definition}

\newtheorem*{df*}{Definition}
\newtheorem*{dfs*}{Definitions}

\newtheorem*{exercise*}{Exercise}

\theoremstyle{remark}
\newtheorem{rem}[thm]{Remark}
\newtheorem*{rem*}{Remark}

\newtheorem*{example}{Example}
\newtheorem*{examples}{Examples}

\newcommand{\sr}[1]{\mathscr{#1}}
\newcommand{\bo}[1]{\boldsymbol{#1}}

\newcommand{\Z}{\mathbf{Z}}

\newcommand{\F}{\mathbf{F}}

\newcommand{\x}{\times}

\newcommand{\pt}{\mathrm{pt}}

\newcommand{\CFhat}{\smash{\widehat{\mathrm{CF}}}}
\newcommand{\HFhat}{\smash{\widehat{\mathrm{HF}}}}
\newcommand{\HFLhat}{\smash{\widehat{\mathrm{HFL}}}}

\DeclareMathOperator{\Id}{Id}

\DeclareMathOperator{\SFH}{SFH}
\DeclareMathOperator{\SFC}{SFC}

\usepackage{etoolbox}
\patchcmd{\thmhead}{(#3)}{#3}{}{}

\newcommand\nomarkerfootnote[1]{%
  \begingroup
  \renewcommand\thefootnote{}\footnote{#1}%
  \addtocounter{footnote}{-1}%
  \endgroup
}

\title{Link Floer homology also detects split links}
\author{Joshua Wang}
\date{}

\begin{document}
\maketitle

\begin{abstract}
	Inspired by work of Lipshitz-Sarkar, we show that the module structure on link Floer homology detects split links. 
	Using results of Ni, Alishahi-Lipshitz, and Lipshitz-Sarkar, we establish an analogous detection result for sutured Floer homology. 
	\nomarkerfootnote{2010 \textit{Mathematics Subject Classification} 57M27}
\end{abstract}

\section{Introduction}

A remarkable feature of modern homology theories in low-dimensional topology is their ability to detect many topological properties of interest. We refer to the introduction of \cite{lipshitz2019khovanov} for a list of such detection results in Khovanov homology and Heegaard Floer homology. The main theorem of \cite{lipshitz2019khovanov} is an additional detection result for Khovanov homology: that the module structure on Khovanov homology detects split links. In this short paper, we add one more item to the list: that the analogous module structure on link Floer homology also detects split links. 

If $L$ is a two-component link in $S^3$, then its link Floer homology $\HFLhat(L)$ \cite{MR2443092}, which takes the form of a finite-dimensional vector space over $\F_2 = \Z/2$, is naturally equipped with an endomorphism $X$ satisfying $X\circ X = 0$. Such an endomorphism gives $\HFLhat(L)$ the structure of a module over $\F_2[X]/X^2$. The map $X$ is defined to be the homological action \cite{MR3294567} of a generator of the first relative homology group of the exterior of $L$. We review the definition of this action in Section~\ref{sec:prelim}. 

\begin{thm}\label{thm:splitLinkDetection}
	Let $L$ be a two-component link in $S^3$. Then $L$ is split if and only if $\HFLhat(L)$ is a free $\F_2[X]/X^2$-module.
\end{thm}

\begin{rem}
	Before the work of this paper began, Tye Lidman observed and informed the author that link Floer homology should detect split links in this way using results of \cite{MR3044606} and arguments similar to those in \cite{lipshitz2019khovanov}. The proof of Theorem~\ref{thm:splitLinkDetection} appearing in this short paper is due to the author and is independent of the results and arguments in \cite{MR3044606} and \cite{lipshitz2019khovanov}. In particular, it uses sutured manifold hierarchies and does not require citing any deep results in symplectic geometry. However, the proof of Theorem~\ref{thm:mainTheorem} appearing here does use such a citation. 
\end{rem}

\begin{rem}
	Link Floer homology detects the Thurston norm of its exterior \cite[Theorem 1.1]{MR2393424}, \cite[Theorem 1.1]{MR2546619}. This by itself does not imply that link Floer homology detects split links. For example, the exterior of the Whitehead link has the same Thurston norm as the exterior of a split union of two genus $1$ knots. 
\end{rem}

More generally, there is a homological action $X_\zeta$ \cite{MR3294567} on the sutured Floer homology $\SFH(M,\gamma)$ \cite{Juh06} of a balanced sutured manifold $(M,\gamma)$ satisfying $X_\zeta \circ X_\zeta = 0$ for each $\zeta \in H_1(M,\partial M)$. 
We prove the following generalization of Theorem~\ref{thm:splitLinkDetection} using a result of Lipshitz-Sarkar \cite{lipshitz2019khovanov}, which builds on results of Alishahi-Lipshitz \cite{MR4010864} and Ni \cite{MR3044606}. 

\begin{thm}\label{thm:mainTheorem}
	Let $(M,\gamma)$ be a balanced sutured manifold, let $\zeta \in H_1(M,\partial M)$, and assume that $\SFH(M,\gamma) \neq 0$. Then $\SFH(M,\gamma)$ is a free $\F_2[X]/X^2$-module with respect to the homological action of $\zeta$ if and only if there is an embedded $2$-sphere $S$ in $M$ for which the algebraic intersection number $S\cdot\zeta$ is odd.
\end{thm}

\begin{rem}
	Let $L$ be a link in $S^3$ with at least two components, and let $C_0,C_1$ be two distinct components of $L$. Let $S^3(L)$ denote the sutured exterior of $L$, and let $\zeta \in H_1(S^3(L),\partial S^3(L))$ be the relative homology class of a path from $C_0$ and $C_1$ in $S^3(L)$. The sutured Floer homology of $S^3(L)$ can be identified with $\HFLhat(L)$ by \cite[Proposition 9.2]{Juh06}. 
	By Theorem~\ref{thm:mainTheorem}, $\HFLhat(L)$ is a free $\F_2[X]/X^2$-module with respect to the homological action of $\zeta$ if and only if there exists an embedded $2$-sphere in the complement of $L$ which separates $C_0$ from $C_1$. 
\end{rem}

Homological actions on Heegaard Floer homology for closed oriented $3$-manifolds were originally defined in \cite[Section 4.2.5]{MR2113019}. These homological actions on Heegaard Floer homology and the closely-related construction of twisted Heegaard Floer homology are well-studied. See \cite{MR2543922,MR2653731,MR3044606,MR3190305,MR4010864,lipshitz2019khovanov,hom2020dehn} for further connections to non-separating spheres and to the module structure on Khovanov homology. 

\theoremstyle{definition}
\newtheorem*{ack}{Acknowledgments}
\begin{ack}
	I thank Tye Lidman, Robert Lipshitz, and Maggie Miller for helpful discussions. I also thank my advisor Peter Kronheimer for his continued guidance, support, and encouragement. This material is based upon work supported by the NSF GRFP through grant DGE-1745303.
\end{ack}

\section{Preliminaries}\label{sec:prelim}

See \cite{Gab83,Gab87a,Juh06,Juh08,MR2653728} for the definitions of balanced sutured manifolds, nice surface decompositions, and sutured Floer homology. In this paper, we use sutured Floer homology with coefficients in the field $\F_2 = \Z/2$. We first provide some examples of sutured manifolds which also serve to set notation. 

\begin{examples}
	A \textit{product sutured manifold} is a sutured manifold of the form $([-1,1] \x \Sigma, [-1,1] \x \partial \Sigma)$ where $\Sigma$ is a compact oriented surface. It is balanced if $\Sigma$ has no closed components.

	Let $L$ be a link in a closed oriented connected $3$-manifold $Y$. The \textit{sutured exterior} $Y(L)$ of the link is the balanced sutured manifold obtained from $Y$ by deleting a regular neighborhood of $L$ and adding two oppositely oriented meridional sutures on each boundary component. 

	If $Y$ is a closed oriented connected $3$-manifold, then let $Y(n)$ be the balanced sutured manifold obtained by deleting $n$ disjoint open balls from $Y$ and adding a suture to each boundary component. We similarly define $(M,\gamma)(n)$ when $(M,\gamma)$ is a connected balanced sutured manifold. 
\end{examples}

\begin{df*}[{\cite[Definition 2.10]{Gab83}}]
	A sutured manifold $(M,\gamma)$ is \textit{taut} if $M$ is irreducible and $R(\gamma)$ is norm-minimizing in $H_2(M,\gamma)$. 
\end{df*}

\begin{rem}
	Balanced product sutured manifolds are taut. If $L$ is a two-component link in $S^3$, then $S^3(L)$ is taut if and only if $L$ is not split. Except for $S^3(1)$, any balanced sutured manifold of the form $Y(n)$ or $(M,\gamma)(n)$ for $n \ge 1$ is not irreducible and therefore is not taut.
\end{rem}

\begin{thm}[{\cite[Theorem 4.2]{Gab83}, \cite[Theorem 8.2]{Juh08}}]\label{thm:niceHierarchy}
	If $(M,\gamma)$ is a taut balanced sutured manifold, then there is a sequence of nice surface decompositions \[
 		(M,\gamma) \overset{S_1}{\rightsquigarrow} (M_1,\gamma_1) \overset{S_2}{\rightsquigarrow} \cdots \overset{S_n}{\rightsquigarrow} (M_n,\gamma_n)
 	\]where $(M_n,\gamma_n)$ is a balanced product sutured manifold. 
\end{thm}

\vspace{10pt}

Let $(M,\gamma)$ be a balanced sutured manifold. We review Ni's definition of the homological action of a relative homology class $\zeta \in H_1(M,\partial M)$ on $\SFH(M,\gamma)$ \cite[Section 2.1]{MR3294567}. These actions are an extension of the homological actions on Heegaard Floer homology defined in \cite[Section 4.2.5]{MR2113019}. We then recall Ni's result that these homological actions are compatible with nice surface decompositions \cite[Theorem 1.1]{MR3294567}.

\begin{df*}[(Homological actions on sutured Floer homology)]
	Let $(\Sigma,\bo\alpha,\bo\beta)$ be an admissible balanced diagram for $(M,\gamma)$, and let $\omega = \sum k_ic_i$ be a formal finite sum of properly embedded oriented curves $c_i$ on $\Sigma$ with integer coefficients $k_i$. Each $c_i$ is required to intersect the $\alpha$- and $\beta$-curves transversely and to be disjoint from every intersection point of the $\alpha$- and $\beta$-curves. Let $\zeta \in H_1(M,\partial M)$ denote the relative homology class that $\omega$ represents. Any relative first homology class of $M$ is represented by such a relative $1$-cycle on $\Sigma$. 

	Let $\SFC(\Sigma,\bo\alpha,\bo\beta)$ be the sutured Floer chain complex whose differential $\partial$ is defined with respect to a suitable family of almost complex structures. In particular \[
		\partial \mathbf{x} = \sum_{\mathbf{y} \in \mathbf{T}_\alpha \cap \mathbf{T}_\beta} \sum_{\substack{\phi \in \pi_2(\mathbf{x},\mathbf{y})\\\mu(\phi) = 1}} \# \widehat{\sr M}(\phi)\cdot\mathbf{y}
	\]for each $\mathbf{x} \in \mathbf{T}_\alpha \cap \mathbf{T}_\beta$. Any Whitney disc $\phi \in \pi_2(\mathbf{x},\mathbf{y})$ has an associated $2$-chain on $\Sigma$ called its domain $D(\phi)$. Let $\partial_\alpha D(\phi)$ be the part of $\partial D(\phi)$ lying in the $\alpha$-circles, thought of as an oriented multi-arc from $\mathbf{x}$ to $\mathbf{y}$. Define $X_\zeta\colon \SFC(\Sigma,\bo\alpha,\bo\beta) \to \SFC(\Sigma,\bo\alpha,\bo\beta)$ by \[
		X_\zeta \cdot \mathbf{x} = \sum_{\mathbf{y} \in \mathbf{T}_\alpha \cap \mathbf{T}_\beta} \sum_{\substack{\phi \in \pi_2(\mathbf{x},\mathbf{y})\\\mu(\phi) = 1}} (\omega \cdot \partial_\alpha D(\phi))\: \# \widehat{\sr M}(\phi)\cdot\mathbf{y}
	\]where $\omega\cdot \partial_\alpha D(\phi)$ is the algebraic intersection number mod $2$. 

	It is shown in \cite{MR3294567,MR2113019} that $X_\zeta$ is a chain map, and its induced map on homology, also denoted $X_\zeta$, squares to zero. Furthermore, the map on homology is independent of the  choice of $\omega$ representing $\zeta \in H_1(M,\partial M)$ and the choice of admissible Heegaard diagram. This induced map on $\SFH(M,\gamma)$ is called the \textit{homological action} of $\zeta$. Unless otherwise stated, $X_\zeta$ refers to the map on homology. Since $X_\zeta\circ X_\zeta = 0$, we may view $\SFH(M,\gamma)$ as a module over $\F_2[X]/X^2$ where the action of $X$ is $X_\zeta$. Note that if $\zeta = 2\zeta'$ is an even homology class, then $X_\zeta = 0$. More generally, $X_\zeta$ is additive in $\zeta$, which is to say that $X_{\zeta + \zeta'} = X_\zeta + X_{\zeta'}$. 
\end{df*}

\begin{example}
	We will make use of the following direct computation. Let $\Sigma$ denote an annulus, and let $\bo\alpha,\bo\beta$ be embedded essential curves which intersect transversely in two points. Then $(\Sigma,\bo\alpha,\bo\beta)$ is an admissible diagram for the balanced sutured manifold $S^3(2)$. Then the two points in the intersection $\mathbf{T}_\alpha \cap \mathbf{T}_\beta$ can be labeled $\mathbf{x},\mathbf{y}$ so that there are two Whitney discs from $\mathbf{x}$ to $\mathbf{y}$. Each has a unique holomorphic representative, and they cancel in the differential of $\SFC(\Sigma,\bo\alpha,\bo\beta)$ so $\dim_{\F_2} \SFH(S^3(2)) = 2$.
	Let $\omega$ be an embedded oriented arc in $\Sigma$ whose endpoints lie on different boundary components and which intersects $\bo\alpha \cup \bo\beta$ transversely in exactly two points. Note that $\omega$ represents a generator $\zeta$ of $H_1(S^3(2),\partial S^3(2))$. Then $X_\zeta \cdot \mathbf{x} = \mathbf{y}$ so $\SFH(S^3(2)) = \F_2[X]/X^2$ as a module with respect to the action of $\zeta$. 

	Essentially the same computation shows that $\HFhat(S^1 \x S^2) = \F_2[X]/X^2$ as a module with respect to the action of a generator of $H_1(S^1 \x S^2)$. 
\end{example}

\begin{rem}
	Let $Y$ be a closed oriented connected $3$-manifold. Then there is an identification $\HFhat(Y) = \SFH(Y(1))$ \cite[Proposition 9.1]{Juh06}. The homological actions on $\HFhat(Y)$ defined in \cite{MR2113019} correspond to the homological actions on $\SFH(Y(1))$ using the natural identification $H_1(Y) = H_1(Y(1),\partial Y(1))$. 
\end{rem}

\begin{rem}\label{rem:disjointUnion}
	Let $(M,\gamma)$ and $(N,\beta)$ be balanced sutured manifolds, and let $\zeta \in H_1(M,\partial M)$ and $\xi \in H_1(N,\partial N)$. Consider the disjoint union $(M \amalg N, \gamma \cup \beta)$ and the relative homology class \[
		\zeta \oplus \xi \in H_1(M,\partial M) \oplus H_1(N,\partial N) = H_1(M \amalg N, \partial (M \amalg N)).
	\]The homological action $X_{\zeta\oplus\xi}$ on $\SFH(M \amalg N,\gamma \cup \beta) = \SFH(M,\gamma) \otimes_{\F_2} \SFH(N,\beta)$ is given by \[
		X_{\zeta\oplus\xi}(a\otimes b) = (X_\zeta a)\otimes b + a \otimes (X_\xi b).
	\]
\end{rem}

\begin{thm}[{\cite[Theorem 1.1]{MR3294567}}]\label{thm:niSurfDecomposition}
	Let $(M,\gamma) \overset{S}{\rightsquigarrow} (M',\gamma')$ be a nice surface decomposition of balanced sutured manifolds. Let $i_*\colon H_1(M,\partial M) \to H_1(M, (\partial M) \cup S) \cong H_1(M',\partial M')$ be the map induced by the inclusion map $i\colon (M,\partial M) \to (M,(\partial M) \cup S)$. 

	View $\SFH(M,\gamma)$ as an $\F_2[X]/X^2$-module with respect to the homological action of some $\zeta \in H_1(M,\partial M)$. View $\SFH(M',\gamma')$ as an $\F_2[X]/X^2$-module with respect to the homological action of $i_*(\zeta) \in H_1(M',\partial M')$. Then $\SFH(M',\gamma')$ is isomorphic as an $\F_2[X]/X^2$-module to a direct summand of $\SFH(M,\gamma)$ as an $\F_2[X]/X^2$-module. 
\end{thm}

The following lemma addresses the special case where the decomposing surface $S$ is a product disc.  
Recall that a decomposing surface $D$ in a sutured manifold $(M,\gamma)$ is called a \textit{product disc} \cite[Definitions 0.1]{Gab87a} if $D$ is a disc and $|D \cap s(\gamma)| = 2$. 

\begin{lem}[{\cite[Lemma 9.13]{Juh06}}]\label{lem:productDiscModuleIso}
	Let $(M,\gamma)$ be a balanced sutured manifold, let $(M,\gamma) \overset{D}{\rightsquigarrow} (M',\gamma')$ be a product disc decomposition, and let $\zeta \in H_1(M,\partial M)$. Then there is an isomorphism $\SFH(M,\gamma) \cong \SFH(M',\gamma')$ of $\F_2[X]/X^2$-modules where $X$ acts on $\SFH(M,\gamma)$ and $\SFH(M',\gamma')$ by the homological actions of $\zeta$ and $i_*(\zeta)$, respectively.
\end{lem}
\begin{proof}
	The lemma follows directly from the definition of the homological action and the proof of \cite[Lemma 9.13]{Juh06}. 
\end{proof}

We will use the following two connected-sum formulas for the homological action on sutured Floer homology. 

\begin{lem}\label{lem:connectSumMandY}
	Let $(M,\gamma)$ be a balanced sutured manifold, and let $Y$ be a closed oriented $3$-manifold. Fix $\zeta \in H_1(M \# Y, \partial (M \# Y))$ and write $\zeta = \zeta' + \zeta''$ according to the decomposition $H_1(M \# Y,\partial (M\#Y)) = H_1(M,\partial M) \oplus H_1(Y)$. Then there is an isomorphism of $\F_2[X]/X^2$-modules \[
		\SFH(M\#Y,\gamma) \cong \SFH(M,\gamma) \otimes_{\F_2} \HFhat(Y)
	\]where $X$ acts on $\SFH(M\#Y,\gamma)$ by $X_\zeta$ while $X$ acts on the right-hand side by \[
		X(a \otimes b) = (X_{\zeta'}a) \otimes b + a \otimes (X_{\zeta''}b).
	\]
\end{lem}
\begin{proof}
	As observed in \cite[Proposition 9.15]{Juh06}, there is a product disc decomposition \[
		(M\# Y, \gamma) \overset{D}{\rightsquigarrow} (M,\gamma) \amalg Y(1).
	\]The result now follows from Remark~\ref{rem:disjointUnion} and Lemma~\ref{lem:productDiscModuleIso}.
\end{proof}

\begin{lem}\label{lem:connectedSumN1andN2}
	Let $(M,\gamma)$ be a connected sum of balanced sutured manifolds $(N_1,\beta_1)$ and $(N_2,\beta_2)$, and let $\zeta \in H_1(M,\partial M)$. Then there is an isomorphism of $\F_2[X]/X^2$-modules \[
		\SFH(M,\gamma) \cong \SFH(N_1,\beta_1) \otimes_{\F_2} \SFH(N_2,\beta_2) \otimes_{\F_2} \SFH(S^3(2))
	\]where $X$ acts on $\SFH(M,\gamma)$ by the homological action of $\zeta$, and $X$ acts on the right-hand side by \[
		X(a \otimes b \otimes c) = (X_{\zeta_1}a) \otimes b \otimes c + a \otimes (X_{\zeta_2}b) \otimes c + a \otimes b \otimes (X_\xi c)
	\]for certain classes $\zeta_1\in H_1(N_1,\partial N_1), \zeta_2 \in H_1(N_2,\partial N_2)$, and $\xi \in H_1(S^3(2),\partial S^3(2))$. 

	Let $S$ be the $2$-sphere in $M$ along which the connected sum is formed. If $S \cdot \zeta$ is odd, then $\SFH(S^3(2)) \cong \F_2[X]/X^2$ with respect to the homological action of $\xi$. If $S \cdot\zeta$ is even, then $X_\xi = 0$. 
\end{lem}
\begin{proof}
	Again as observed in \cite[Proposition 9.15]{Juh06}, there are product disc decompositions \[
		(M,\gamma) \overset{D}{\rightsquigarrow} (N_1,\beta_1) \amalg (N_2,\beta_2)(1) \overset{D'}{\rightsquigarrow} (N_1,\beta_1) \amalg (N_2,\beta_2) \amalg S^3(2)
	\]Let $\zeta_1 \oplus \zeta_2 \oplus \xi$ be the image of $\zeta$ in \[
		H_1(M_1,\partial M_1) \oplus H_1(M_2,\partial M_2) \oplus H_1(S^3(2),\partial (S^3(2))).
	\]Note that $\xi$ is an even class if and only if $S \cdot \zeta$ is even. Furthermore, a direct computation gives an isomorphism $\SFH(S^3(2)) \cong \F_2[X]/X^2$ as modules when $X$ is the homological action of a generator. The result now follows from Remark~\ref{rem:disjointUnion} and Lemma~\ref{lem:productDiscModuleIso}.  
\end{proof}

\vspace{10pt}

Before turning to the main results, we record here an algebraic lemma. 

\begin{lem}\label{lem:freenessOfTensorProduct}
	Let $M_1,\ldots,M_k$ be a collection of finitely-generated $\F_2[X]/X^2$-modules, and view $M = M_1 \otimes_{\F_2} \cdots \otimes_{\F_2} M_k$ as an $\F_2[X]/X^2$-module where the action of $X$ is defined by \[
		X(m_1 \otimes \cdots \otimes m_k) = \sum_{i=1}^k m_1 \otimes \cdots \otimes m_{i-1} \otimes (Xm_i) \otimes m_{i+1} \otimes \cdots \otimes m_k.
	\]Then $M$ is free if and only if at least one of the $M_i$ is free. 
\end{lem}
\begin{proof}
	View $\F_2$ as an $\F_2[X]/X^2$-module where $X$ acts by zero. Then any finitely-generated $\F_2[X]/X^2$-module is isomorphic to $(\F_2[X]/X^2)^{\oplus n} \oplus (\F_2)^{\oplus m}$ for some nonnegative integers $n,m$. The result follows from the computations that \[
		\frac{\F_2[X]}{X^2} \otimes_{\F_2} \frac{\F_2[X]}{X^2} \qquad\text{ and }\qquad \frac{\F_2[X]}{X^2} \otimes_{\F_2} \F_2
	\]are free while $\F_2 \otimes_{\F_2} \F_2$ is not. 
\end{proof}

\section{Main results}

The following lemma contains the main argument of this short paper. Using this lemma, we provide a quick proof of Theorem~\ref{thm:splitLinkDetection} before turning to its generalization. 

\begin{lem}\label{lem:tautImpliesNotFree}
	Let $(M,\gamma)$ be a taut balanced sutured manifold with $\zeta \in H_1(M,\partial M)$. Then $\SFH(M,\gamma)$ is not a free $\F_2[X]/X^2$-module with respect to the homological action of $\zeta$. 
\end{lem}
\begin{proof}
	By Theorem~\ref{thm:niceHierarchy}, we may find a sequence of nice surface decompositions from $(M,\gamma)$ to a product sutured manifold $(N,\beta)$. Then $\SFH(N,\beta)$ is isomorphic to a direct summand of $\SFH(M,\gamma)$ as an $\F_2[X]/X^2$-module by Theorem~\ref{thm:niSurfDecomposition} where $X$ acts on $\SFH(N,\beta)$ by some homological action. Since $\dim_{\F_2} \SFH(N,\beta) = 1$ by \cite[Proposition 9.4]{Juh06}, it cannot be a free $\F_2[X]/X^2$-module. Thus $\SFH(M,\gamma)$ is not a free $\F_2[X]/X^2$-module. 
\end{proof}

\begin{proof}[Proof of Theorem~\ref{thm:splitLinkDetection}]
	Recall that $\HFLhat(L) = \SFH(S^3(L))$ \cite[Proposition 9.2]{Juh06} where $S^3(L)$ is the sutured exterior of $L$. If $L$ is split, then $\SFH(S^3(L))$ is a free $\F_2[X]/X^2$-module by a computation in a carefully chosen Heegaard diagram. This computation can be formalized in the following way. If $L$ is the split union of the knots $K$ and $J$, then $S^3(L) = S^3(K) \:\#\: S^3(J)$. By Lemma~\ref{lem:connectedSumN1andN2}, there is an isomorphism of $\F_2[X]/X^2$-modules \[
		\SFH(S^3(L)) \cong \SFH(S^3(K)) \otimes_{\F_2} \SFH(S^3(J)) \otimes_{\F_2} \SFH(S^3(2))
	\]where the action of $X$ on the right-hand side is given by $\Id \otimes \Id \otimes \,X_\xi$ where $\xi$ is a generator of $H_1(S^3(2),\partial S^3(2))$. Since $\SFH(S^3(2)) \cong \F_2[X]/X^2$ as modules with respect to the action of $\xi$, it follows from Lemma~\ref{lem:freenessOfTensorProduct} that $\SFH(S^3(L))$ is a free $\F_2[X]/X^2$-module. 

	If $L$ is not split, then $S^3(L)$ is taut, so $\SFH(S^3(L))$ is not free by Lemma~\ref{lem:tautImpliesNotFree}. 
\end{proof}

\vspace{10pt}

The next lemma is a direct consequence of \cite[Lemma 5.6]{lipshitz2019khovanov}, which Lipshitz-Sarkar prove using \cite[Theorem 1.1]{MR4010864}. This result of Alishahi-Lipshitz builds on work of Ni \cite{MR3044606}. See also \cite[Theorem 4 and Corollary 5.2]{MR3190305}. 

\begin{lem}\label{lem:closedYnotFree}
	Let $Y$ be an irreducible closed oriented $3$-manifold with $\zeta \in H_1(Y)$. Then $\HFhat(Y)$ is not a free $\F_2[X]/X^2$-module with respect to the homological action of $\zeta$. 
\end{lem}
\begin{proof}
	We use the notation in \cite{lipshitz2019khovanov} without reintroducing it. 
	Since $S \cdot \zeta = 0$ for all embedded $2$-spheres $S$ in $Y$, the unrolled homology of $\CFhat(Y)$ with respect to $\zeta$ is nontrivial by \cite[Lemma 5.6]{lipshitz2019khovanov}. The $E^1$-page of the spectral sequence associated to the horizontal filtration on the unrolled complex of $\CFhat(Y)$ is the unrolled complex of $\HFhat(Y)$ viewed as an $\F_2[X]/X^2$-module with respect to $\zeta$. Since the $E^\infty$-page is nonzero, the $E^2$-page is also nonzero so $\HFhat(Y)$ is not a free $\F_2[X]/X^2$-module. 
\end{proof}

\begin{proof}[Proof of Theorem~\ref{thm:mainTheorem}]
	We first prove the result under the assumption that $(M,\gamma)$ is strongly-balanced. We then prove the general statement by reducing to this case. A balanced sutured manifold $(M,\gamma)$ is \textit{strongly-balanced} \cite[Definition 3.5]{Juh08} if for each component $F$ of $\partial M$, we have the equality $\chi(F \cap R_+(\gamma)) = \chi(F \cap R_-(\gamma))$. 

	Under the assumption that $(M,\gamma)$ is strongly-balanced, suppose there exists a $2$-sphere $S$ in $M$ for which $S \cdot \zeta$ is odd. There are two cases. \begin{enumerate}
		\item[(a)] The sphere $S$ is non-separating. 

		Then $(M,\gamma)$ is the connected sum of a strongly-balanced sutured manifold $(N,\beta)$ with $S^1 \x S^2$, where $S$ is a copy of $\pt \x S^2$ in the $S^1 \x S^2$ summand. If $\zeta = \zeta' \oplus \zeta''$ under the natural identification $H_1(M,\partial M) = H_1(N,\partial N) \oplus H_1(S^1 \x S^2)$, then $\zeta''$ is an odd multiple of a generator $\xi$ of $H_1(S^1 \x S^2)$ by the assumption that $S \cdot \zeta$ is odd. By Lemma~\ref{lem:connectSumMandY}, there is an isomorphism of $\F_2[X]/X^2$-modules \[
			\SFH(M,\gamma) \cong \SFH(N,\beta) \otimes_{\F_2} \SFH(S^1 \x S^2),
		\]where the actions of $X$ on the right-hand side is $X_{\zeta'} \otimes \Id + \Id \otimes\, X_{\zeta''}$. Since $\zeta'' - \xi$ is even, we know that $X_{\zeta''} = X_{\xi}$. By a direct computation, $\SFH(S^1 \x S^2) \cong \F_2[X]/X^2$ as an $\F_2[X]/X^2$-module with respect to the action of $\xi$. It follows that $\SFH(M,\gamma)$ is free by Lemma~\ref{lem:freenessOfTensorProduct}.
		\item[(b)] The sphere $S$ is separating. 

		Then $(M,\gamma)$ is the connected sum of sutured manifolds $(N_1,\beta_1)$ and $(N_2,\beta_2)$ along the sphere $S$. Since $(M,\gamma)$ is strongly-balanced, both $(N_1,\beta_1),(N_2,\beta_2)$ are as well. 
		By Lemma~\ref{lem:connectedSumN1andN2}, there is an isomorphism of $\F_2[X]/X^2$-modules \[
			\SFH(M,\gamma) \cong \SFH(N_1,\beta_1) \otimes_{\F_2} \SFH(N_2,\beta_2) \otimes_{\F_2} \SFH(S^3(2))
		\]where the action of $X$ on the right-hand side is given by \[
			X_{\zeta_1} \otimes \Id \otimes \Id + \Id \otimes\, X_{\zeta_2} \otimes \Id + \Id \otimes \Id \otimes\, X_\xi
		\]for classes $\zeta_i \in H_1(N_i,\partial N_i)$ and $\xi \in H_1(S^3(2),\partial S^3(2))$. Furthermore, by the assumption that $S \cdot \zeta$ is odd, Lemma~\ref{lem:connectedSumN1andN2} implies that $\SFH(S^3(2)) \cong \F_2[X]/X^2$ as modules with respect to the action of $\xi$. Thus $\SFH(M,\gamma)$ is a free $\F_2[X]/X^2$-module by Lemma~\ref{lem:freenessOfTensorProduct}. 
	\end{enumerate}

	Now assume that $S\cdot\zeta$ is even for every embedded $2$-sphere $S$ in $M$, where $(M,\gamma)$ is strongly-balanced. To show that $\SFH(M,\gamma)$ is not a free $\F_2[X]/X^2$-module with respect to the action of $\zeta$, we use Lemmas~\ref{lem:connectSumMandY}, \ref{lem:connectedSumN1andN2}, and \ref{lem:freenessOfTensorProduct} to reduce to the irreducible case, which is then handled by Lemmas~\ref{lem:tautImpliesNotFree} and \ref{lem:closedYnotFree}.
	Write $M$ as a connected sum \[
		M = N_1 \:\#\: \cdots \:\#\: N_k \:\#\: Y_1 \:\#\: \cdots \:\#\: Y_\ell \:\#\: (S^1 \x S^2)^{\#m}
	\]where the $N_i$ are irreducible compact $3$-manifolds with nonempty boundary and the $Y_i$ are irreducible closed $3$-manifolds. A quick way to see that such a decomposition exists uses the Grushko-Neumann theorem on the ranks of free products of finitely-generated groups and the Poincar\'e conjecture (for example, see \cite{MR142125}). The assumption that $(M,\gamma)$ is strongly-balanced implies $(N_i,\beta_i)$ is also strongly-balanced where $\beta_i$ are the sutures inherited from $\gamma$. Because $\SFH(M,\gamma) \neq 0$, it follows that $\SFH(N_i,\beta_i) \neq 0$ so $(N_i,\beta_i)$ is taut by \cite[Proposition 9.18]{Juh06}. Note that $S\cdot \zeta$ is even for each sphere along which the connected sums are formed and for each sphere of the form $\pt \x S^2$ in each $S^1 \x S^2$ summand. The result now follows from Lemmas~\ref{lem:connectSumMandY}, \ref{lem:connectedSumN1andN2}, \ref{lem:freenessOfTensorProduct}, \ref{lem:tautImpliesNotFree}, and \ref{lem:closedYnotFree}.

	We now reduce to the case that $(M,\gamma)$ is strongly-balanced. As explained in \cite[Remark 3.6]{Juh08}, we may construct a strongly-balanced sutured manifold $(M',\gamma')$ from a given balanced sutured manifold $(M,\gamma)$ so that there is a sequence of product disc decompositions from $(M',\gamma')$ to $(M,\gamma)$. To construct $(M',\gamma')$, we repeat the following procedure, which is sometimes referred to as a contact $1$-handle attachment: fix two components $F_1,F_2$ of $\partial M$ for which $\chi(F_i \cap R_+(\gamma)) \neq \chi(F_i \cap R_-(\gamma))$, choose discs $D_i$ centered at points on $s(\gamma) \cap F_i$, and identify $D_1$ with $D_2$ by an orientation-reversing map so that the sutures $s(\gamma) \cap D_1$ and $s(\gamma) \cap D_2$ are identified. Do this identification in such a way that the orientations of the sutures are reversed. The resulting manifold naturally inherits an orientation and sutures for which there is at least one fewer boundary component $F$ with $\chi(F \cap R_+(\gamma)) \neq \chi(F \cap R_-(\gamma))$. This reverse of this procedure is a product disc decomposition along $D_1 = D_2$. 

	Let $(M,\gamma)$ be a balanced sutured manifold, and let $\zeta \in H_1(M,\partial M)$. Construct a strongly-balanced sutured manifold $(M',\gamma')$ from $(M,\gamma)$ as explained. Let $z \subset M$ be a properly-embedded oriented $1$-manifold representing $\zeta$ which is disjoint from the discs on the boundary used in the construction of $(M',\gamma')$. Then $z$ represents a class $\zeta' \in H_1(M',\partial M')$ for which $i_*(\zeta') = \zeta$ under the sequence of product disc decompositions from $(M',\gamma')$ to $(M,\gamma)$. Note that $\SFH(M,\gamma)$ and $\SFH(M',\gamma')$ are isomorphic as $\F_2[X]/X^2$-modules with respect to the actions of $\zeta$ and $\zeta'$, respectively, by Lemma~\ref{lem:productDiscModuleIso}.

	If there is an embedded $2$-sphere $S$ in $M$ with $S \cdot \zeta$ odd, then the same $2$-sphere viewed in $M'$ also has $S \cdot \zeta'$ odd. Thus $\SFH(M',\gamma')$ and $\SFH(M,\gamma)$ are free. Conversely, if $\SFH(M,\gamma)$ is free, then $\SFH(M',\gamma')$ is free as well, so there is an embedded $2$-sphere $S'$ in $(M',\gamma')$ which intersects $z$ transversely in an odd number of points. By using an innermost argument, we may compress $S'$ along discs in the product discs of the sequence of decompositions from $(M',\gamma')$ to $(M,\gamma)$ to obtain a collection of embedded $2$-spheres $S_i$ in $M$ for which $\sum_i S_i \cdot \zeta = S' \cdot \zeta'$ is odd. Thus there is at least one $S_i$ for which $S_i\cdot \zeta$ is odd.
\end{proof}

\raggedright
\bibliography{FloerSplitLinks}
\bibliographystyle{alpha}

\vspace{10pt}

\textit{Department of Mathematics}

\textit{Harvard University}

\textit{Science Center, 1 Oxford Street}

\textit{Cambridge, MA 02138}

\textit{USA}

\vspace{10pt}

\textit{Email:} \texttt{jxwang@math.harvard.edu}

\end{document}